\documentclass[a4paper, 11pt]{article}
\title{Isotopy of Morin singularities}
\usepackage{amssymb}
\usepackage{amsmath}
\usepackage{amsthm}
 \newtheorem{theorem}{Theorem}[section]

 \newtheorem{lemma}[theorem]{Lemma}
 \newtheorem{corollary}[theorem]{Corollary}
\theoremstyle{definition}
 \newtheorem{definition}[theorem]{Definition}

\numberwithin{equation}{section}
\numberwithin{figure}{section}
\numberwithin{table}{section}
\newcommand{\R}{\boldsymbol{R}}
\newcommand{\C}{\boldsymbol{C}}
\newcommand{\Z}{\boldsymbol{Z}}
\newcommand{\N}{\boldsymbol{N}}
\newcommand{\trace}{\operatorname{trace}}
\newcommand{\rank}{\operatorname{rank}}
\newcommand{\hess}{\operatorname{Hess}}
\newcommand{\grad}{\operatorname{grad}}

\renewcommand{\phi}{\varphi}
\newcommand{\sign}{\operatorname{sgn}}

\newcommand{\spann}[1]{\left\langle{#1}\right\rangle}

\newcommand{\A}{\mathcal{A}}
\newcommand{\E}{\mathcal{E}}
\newcommand{\ep}{\varepsilon}
\newcommand{\zv}{0}
\newcommand{\pmt}[1]{{\begin{pmatrix} #1  \end{pmatrix}}}
\newcommand{\trans}[1]{{\vphantom{#1}}^t{\!#1}}
\newcommand{\diff}[2]{\operatorname{Diff}^{#2}(#1)}
\newcommand{\mycomment}[1]{}
\makeatletter
\renewcommand{\section}{%
  \@startsection{section}% #1 name
   {1}% #2 level
   {\z@}% #3 indent length
   {-3.5ex \@plus -1ex \@minus -.2ex}% #4 up space
   {2.3ex \@plus.2ex}% #5 under space
   {\normalfont\normalsize\bfseries}% #6
}%
\makeatother%
\makeatletter
\renewcommand{\subsection}{%
  \@startsection{subsection}% #1 name
   {1}% #2 level
   {\z@}% #3 indent length
   {-3.5ex \@plus -1ex \@minus -.2ex}% #4 up space
   {2.3ex \@plus.2ex}% #5 under space
   {\normalfont\normalsize\bfseries}% #6
}%
\makeatother
\author{Kentaro Saji}
\date{\today}
\begin{document}
\maketitle
 \renewcommand{\thefootnote}{\fnsymbol{footnote}}
 \footnote[0]{Keywords and Phrase.
 Morin singularities, deformation}
 \footnote[0]{Dedicated to Professor Stanis{\l}aw Janeczko on 
  the occasion of his sixtieth birthday}
\begin{abstract}
We define an equivalence relation called
$\A$-isotopy between fini\-te\-ly
determined map-germs, which is a 
strengthened version of $\A$-equi\-va\-lence.
We consider the number of $\A$-isotopy classes
of equidimensional Morin singularities,
and some other well-known low-dimensional
singularities. We also give an application to stable
perturbations of simple equi-dimensional map-germs.
\end{abstract}
\section{Introduction}
There are various groups which act on the set $C^\infty(m,n)$
of map-germs $(\R^m,0)\to(\R^n,0)$.
The group $\A$ will denote the group of changes of coordinates in the source 
and target, which act on $C^\infty(m,n)$ by $\tau\circ f\circ \sigma$,
where $f\in C^\infty(m,n)$ and $\sigma$ and $\tau$ are,
respectively, diffeomorphism-germs in the source and target.
Two map-germs $f,g\in C^\infty(m,n)$ are {\em $\A$-equivalent}\/
if they belong to the same orbit.
In this paper, we define an equivalence relation 
called {\em $\A$-isotopy}, which is a 
strengthened version of $\A$-equivalence.
Let $r$ be a natural number.
A map-germ $f\in C^\infty(m,n)$ is said to be
{\em $r$-determined}\/ if 
any $g\in C^\infty(m,n)$ satisfying
$j^rf(0)=j^rg(0)$ is
$\A$-equivalent to $f$,
where $j^rf(0)$ is the $r$-jet of $f$ at $0$.
%A $r$-jet $j^rf(0)$ of $f$ is {\em $r$-sufficient}\/
%if $f$ is $r$-determined.
\begin{definition}
Let $f,g\in C^\infty(m,n)$ be $\A$-equivalent
map-germs that are $r$-determined.
Then $f$ and $g$ are {\em $\A$-isotopic}\/
if there exist continuous curves
$\sigma:I\to\diff{m}{r}\subset J^r(m,m)$ and
$\tau:I\to\diff{n}{r}\subset J^r(n,n)$ such that
$\sigma(0)$, $\tau(0)$ are both the identity, and
$$j^r(g)(0)=j^r\big(\tau(1)\circ f\circ \sigma(1)\big)(0)$$
holds, where $I=[0,1]$ and
$\diff{m}{r}$ denotes the set of the $r$-jets of
diffeo\-mor\-phism-germs
$(\R^m,0)\to(\R^m,0)$.
\end{definition}
Namely, $f$ and $g$ are $\A$-isotopic
if and only if
$j^rf(0)$ and $j^r g(0)$ are located on
the same arc-wise connected component of
the $r$-jet of the $\A$-orbit of $j^rf(0)$.
Since the set $\diff{m}{r,+}$ of $r$-jets
of orientation-preserving diffeomorphism-germs
is arc-wise connected,
%and the identity is an element of $\diff{m}{r,+}$,
$f$ and $g$ are $\A$-isotopic
if and only if
there exist
orientation-preserving diffeomorphism-germs
$\sigma^+:(\R^m,0)\to(\R^m,0)$ and
$\tau^+:(\R^n,0)\to(\R^n,0)$
such that
$j^rg(0)=j^r(\tau^+\circ f\circ \sigma^+)(0)$ holds.

In this paper, we study the number of $\A$-isotopy 
classes of equidimensional Morin singularities.
Morin singularities are stable, and conversely, corank one and stable
germs are Morin singularities. 
This means that Morin singularities are fundamental and
frequently appear as singularities of
maps from a manifold to another.
We show that
for an $n$-Morin singularity $f:(\R^n,0)\to(\R^n,0)$,
there are four (respectively, two)
$\A$-isotopy 
classes in $\A(f)$
if $n=4i$ (respectively, $n\ne 4i$) with $i\in\N$ 
(see section \ref{sec:isotopy}),
where $\A(f)$ stands for the $\A$-orbit of $f$.
For a $k$-Morin singularity $f:(\R^n,0)\to(\R^n,0)$ ($k<n$),
there are two 
$\A$-isotopy 
classes (respectively, is one $\A$-isotopy class) in $\A(f)$
if $k=2i$ (respectively, $k\ne 2i$) (see section \ref{sec:isotopy}).
The tables in section \ref{sec:normalform}
summerize the invariants and normal forms for the $\A$-isotopy classes
of these Morin singularities.
In section \ref{sec:other}, we consider the same problem for
some other well-known low-dimensional singularities.
As an application, we consider in section \ref{sec:perturb} $\A$-isotopy 
classes of $n$-Morin singularities appearing on stable perturbations
of simple map-germs $(\R^n,0)\to(\R^n,0)$.
We remark that homotopy types of the $\mathcal{A}^2$-orbit of
the fold are
considered by Ando \cite{ando2,ando3},
thus we are mainly interested in the case $k\geq2$.
We also remark that this type of problem is asked
by Nishimura \cite{nfbook}.

\section{$\A$-isotopy of Morin singularities}\label{sec:isotopy}
Let $f\in C^\infty(n,n)$ be an equidimensional map-germ.
Then $f$ is a {\em Morin singularity of\/ $\Sigma^{(1^k,0)}$-type}\/
(or shortly, a {\em $k$-Morin singularity}\/)
if $f$ is $\A$-equivalent to
the germ
$$
(x_1,\ldots,%x_k,x_{k+1},\ldots,
x_n)
\mapsto
(x_1x_2+x_1^2x_3+\cdots+x_1^{k-1}x_{k}+x_1^{k+1},\ 
x_2,\ \ldots,\ %x_k,x_{k+1},\ \ldots,
x_n)
$$
at the origin, where $k\leq n$.
For the meaning of the notation, and further details,
see \cite{morin}.
It is well known that a $k$-Morin singularity
is $(k+1)$-determined.
There are recognition criteria for $k$-Morin singularities \cite{SUY3}.
Let $f\in C^\infty(n,n)$ and 
$\lambda$ be the determinant of the Jacobi matrix of $f$.
Let $0$ be a singular point of $f$, namely $\lambda(0)=0$,
then the singular point $0$ is {\em non-degenerate}\/ 
if $d\lambda(0)\ne0$.
Let $0$ be a non-degenerate singular point of $f$,
then there exists a never-vanishing vector field $\eta$
around $0$ on $\R^n$
such that
$\eta(p)\in\ker df(p)$
for $p\in S(f)$,
where $S(f)$ is the set of singular points of $f$.
We call $\eta$ the {\em null-vector field}.
Then the following theorem holds.
\begin{theorem}{\rm \cite[Theorem A1, page 746]{SUY3}}
Let\/ $f$ and\/ $\lambda$ be as above.
Then\/ $f$ at\/ $0$ is a\/ $k$-Morin singularity
if and only if
\begin{equation}\label{eq:morinhanti}
\begin{array}{l}
\lambda(0)=\eta\lambda(0)=\cdots=\eta^{k-1}\lambda(0)=0,\ 
\eta^{k}\lambda(0)\ne0\\
\hspace{40mm}{\rm and}\quad
\rank
d(\lambda,\eta\lambda,\ldots,\eta^{k-1}\lambda)(0)
=k
\end{array}
\end{equation}
hold, where\/ $d(\lambda,\eta\lambda,\ldots,\eta^{k-1}\lambda)$ 
denotes the differential of the map\/ 
$$(\lambda,\eta\lambda,\ldots,\eta^{k-1}\lambda):(\R^n,0)\to(\R^k,0),$$
and\/ $\eta\lambda$ denotes the directional derivative of\/ $\lambda$
with respect to\/ $\eta$,
and\/ $\eta^k\lambda=
\underbrace{\eta\cdots\eta}_{k\text{-times}}\lambda$.
\end{theorem}
%We remark that if $k=n$, the last condition
%of \eqref{eq:morinhanti} is equivalent to 
%\begin{equation}\label{eq:morinhantidet}
%\det
%d(\lambda,\eta\lambda,\ldots,\eta^{k-1}\lambda)(0)
%\ne 0.
%\end{equation}
We have the following lemma.
\begin{lemma}\label{lem:prenormal}
Let\/ $f\in C^\infty(n,n)$ be a\/ $k$-Morin singularity.
Assume\/ $k\ne1$.
Then\/ $f$ is\/ $\A$-isotopic to
\begin{equation}\label{eq:morinconn}
f_{(\ep_1,\ep_2)}^k(x)
=\Big(
\ep_1\big(
\ep_2x_2x_1+x_3x_1^2+\cdots+x_{k}x_1^{k-1}+x_1^{k+1}
\big),\ 
\ep_2x_2,\ x_3,\ \ldots,\ x_n\Big),
\end{equation}
where $x=(x_1,\ldots,x_n)$, 
$\ep_1=\pm1$ and\/ $\ep_2=\pm1$.
If\/ $k=1$, 
then\/ $f$ is\/ $\A$-isotopic to\/
$f_{\ep_1}^1=(\ep_1x_1^2,x_2,\ldots,x_n)$
where\/ $\ep_1=\pm1$.
\end{lemma}
In what follows, we use the following notation:
For a given map-germ $(\R^n,0)\to(\R^n,0)$,
the small letters
$x=(x_1,\ldots,x_n)$ denote the coordinate system on the source space,
and the capital letters
$X=(X_1,\ldots,X_n)$ denote that of the target space.
Following a characterization of
Morin singularities given in \cite{morin},
and taking care to use only orientation-preserving 
diffeomorphism-germs, one can easily prove Lemma 
\ref{lem:prenormal}.
\begin{proof}[Proof of Lemma\/ {\rm \ref{lem:prenormal}}.]
Assume that 
$f(x)=(f_1(x),\ldots,f_n(x))$
is a $k$-Morin singularity.
Since $d\lambda(0)\ne0$, we have $\rank (df)(0)=n-1$.
Then by a rotation, we may assume that
$$
\left\{
\trans{(1,0,\ldots,0)},\ \grad(f_2)(0),\ \ldots,\ \grad(f_n)(0)
\right\}
$$
forms a positive
basis of $\R^n$,
where
$\grad(h)=dh=\trans{(h_{x_1},\ldots,h_{x_n})}$
and $h_{x_1}=\partial h/\partial x_1$, for example,
and $\trans{(\cdot)}$ means the transpose matrix.
Moreover, since $\{A\in O(n)\,|\,\det A>0\}$ is arc-wise connected,
we may assume that
$$\Big(\trans{(1,0,\ldots,0)},\grad(f_2)(0),\ldots,\grad(f_n)(0)\Big)=E,$$
where $E$ is the identity matrix.
Then the map-germ
$
x
\mapsto
\big(
x_1,
f_2(x),\ldots,
f_n(x)
\big)
$
is an orientation-preserving diffeomorphism-germ.
Hence we may assume that
$$
f(x_1,\ldots,x_n)
=\big(
f_1(x_1,\ldots,x_n),
x_2,\ldots,x_n\big).
$$
Then we can take the null vector field
$\eta=\partial_{x_1}$.
Since 
$\lambda=(f_1)_{x_1}$ and
$\eta^k\lambda(0)\ne0$,
it holds that
$f_1(x_1,0,\ldots,0)=ax_1^{k+1}+\cdots$ $(a\ne0)$.
Then by the Malgrange preparation theorem,
there exist functions $a_0,\ldots,a_k$
of $(X_1,\ldots,X_n)$ 
such that
\begin{equation}
\label{eq:mal1}
x_1^{k+1}
=
a_0\big(f(x)\big)-\Big(a_1\big(f(x)\big)x_1
+\cdots+a_{k}\big(f(x)\big)x_1^k\Big)\quad
\big(x=(x_1,\ldots,x_n)\big)
\end{equation}
holds.
Considering an orientation-preserving diffeomorphism-germ
$$
\phi(x)
=
\Big(
x_1+\frac{1}{k}a_{k}\big(f(x)\big),\ 
x_2,\ \ldots,\ x_n
\Big)
$$
and set $\tilde x=(\tilde x_1,\ldots,\tilde x_n)
=\phi(x)$.
Then by a direct calculation, there exist
functions $b_0,\ldots,b_{k-1}$ such that
\begin{equation}
\label{eq:mal2}
\tilde x_1^{k+1}
=
b_0\Big(f\circ\phi^{-1}(\tilde x)\Big)-
\sum_{i=1}^{k-1}\tilde x_1^{i} b_{i}\Big(f\circ\phi^{-1}(\tilde x)\Big).
\end{equation}
Differentiating the equation \eqref{eq:mal2} by $\tilde x_1$,
we see that $b_0(0)=\cdots =b_{k-1}(0)=0$.
Furthermore, setting $\tilde x_2=\cdots =\tilde x_k=0$ in \eqref{eq:mal2}
and expanding both sides in powers of $\tilde x_1$,
we see that
$(b_1)_{X_1}(0)\ne0$.
Thus 
$$
\Phi(X)
=
\big(\ep_0b_0(X),X_2\ldots,X_n\big)\quad
\big(X=(X_1\ldots,X_n)\big)
$$
is an orientation-preserving diffeomorphism-germ,
where, $\ep_0=\sign((b_1)_{X_1}(0))=\pm1$.
Then we see that $f$ is $\A$-isotopic to
\begin{equation}\label{eq:f1}
\Phi\circ f\circ \phi^{-1}(\tilde x)
=
\Bigg(
\ep_0\bigg\{
\tilde x_1^{k+1}
-\sum_{i=1}^{k-1}
\tilde x_1^{i} b_{i}\Big(f\circ\phi^{-1}(\tilde x)\Big)
\bigg\},\ 
\tilde x_2,\ \ldots,\ \tilde x_n\Bigg).
\end{equation}
If $k=1$, we have the assertion.
We assume $k>1$ in what follows.
Since the condition
$\rank d(\lambda,\eta\lambda,\ldots,\eta^{k-1}\lambda)(0)
=k$ does not depend on the coordinate system,
we may assume that 
$\lambda=(\partial/\partial \tilde x_1)\big\{
\tilde x_1^{k+1}
-\sum_{i=1}^{k-1}
\tilde x_1^{i} b_{i}\big(f\circ\phi^{-1}(\tilde x)\big)
\big\}$ and $\eta=\partial \tilde x_1$.
Thus we see that
$$
\grad\bigg(b_1\Big(f\circ\phi^{-1}(\tilde x)\Big)\bigg)(0),\ \ldots,\ 
\grad\bigg(b_{k-1}\Big(f\circ\phi^{-1}(\tilde x)\Big)\bigg)(0),\ 
\grad \tilde x_1(0)
$$
are linearly independent.
Thus
$$
\begin{array}{rcl}
\psi(\tilde x)&=&
\bigg(\tilde x_1,\ 
\ep_2b_1\Big(f\circ\phi^{-1}(\tilde x)\Big),\ 
\ldots,
b_{k-1}\Big(f\circ\phi^{-1}(\tilde x)\Big),\ 
\tilde x_{k+1},\ \ldots,\ \tilde x_n\bigg)\\
\Psi(X)
&=&
\big(\ep_3b_0(X),\ b_1(X),\ \ldots,\ b_{k-1}(X),\ 
X_{k+1},\ \ldots,\ X_n\big)
\end{array}
$$
are orientation-preserving diffeomorphism-germs
for some $\ep_2=\pm1$ and $\ep_3=\pm1$,
and we see that $f$ is $\A$-isotopic to
$\Psi\circ f\circ \phi^{-1}\circ \psi^{-1}$.
Setting $\ep_1=\ep_0\ep_3$, 
we complete the proof.
\end{proof}
Let $f$ be a map-germ
of the form
\eqref{eq:morinconn}.
Since it holds that
$$\lambda=
\ep_1\ep_2
\dfrac{\partial}{\partial x_1}
\big(\ep_2x_2x_1+x_3x_1^2+\cdots+x_{k}x_1^{k-1}+x_1^{k+1}\big),
\quad
\eta=\partial x_1,
$$
we have
\begin{equation}
\label{eq:sgnlambda}
\sign\big(\eta^{k}\lambda(0)\big)=\ep_1\ep_2,\quad
\sign\det \grad(\lambda,\eta\lambda,\ldots,\eta^{n-1}\lambda)(0)
=
(-1)^{n-1}\ep_1^n\ep_2^{n+1},
\end{equation}
where
$
\grad(\lambda,\eta\lambda,\ldots,\eta^{n-1}\lambda)
=
(\grad \lambda,\grad \eta\lambda,\ldots,\grad \eta^{n-1}\lambda).
$
By orientation-preserving diff\-eo\-mor\-phism-germs
on source and target,
$\lambda$ is multiplied by a positive function.
On the other hand,
reversing the direction of $\eta$, the sign of 
$\eta\lambda$ changes.
Summarizing the above arguments,
we have the following lemma.
\begin{lemma}\label{lem:sgninv}
Let\/ $f\in C^\infty(n,n)$ be a\/ $k$-Morin singularity.
If\/ $k$ is even, then\/
$\sign(\eta^{k}\lambda)=\ep_1\ep_2$ is an invariant
of\/ $\A$-isotopy.
If\/ $k=n$ and\/ $1+\cdots+n-1=(n-1)n/2$ is even,
then\/ $\sign\det d(\lambda,\ldots,\eta^{n-1}\lambda)=
(-1)^{n-1}\ep_1^n\ep_2^{n+1}$ 
is an invariant
of\/ $\A$-isotopy.
Furthermore, if\/ $k=n$ and,
$n$ and\/ $(n-1)n/2$ are both odd,
then
$$\sign\Big(\eta^{k}\lambda\cdot
\det \grad( \lambda,\eta\lambda,\ldots,\eta^{k-1}\lambda)\Big)=
(-1)^{n-1}\ep_1^{n+1}\ep_2^n$$ 
is an invariant
of\/ $\A$-isotopy.
\end{lemma}
Now we consider $\A$-isotopy of 
$f_{(\ep_1,\ep_2)}^k(x_1,\ldots,x_n)$.
By the above lemma, in
the case of $k=n$, 
we consider four cases 
$k=n=4l,\ 4l+1,\ 4l+2,\ 4l+3$.

\subsection{The case $k=n=4l$}
By Lemma \ref{lem:sgninv}, we see that
$\ep_1\ep_2$ and
$-\ep_1^{4l}\ep_2^{4l+1}=-\ep_2$
are invariants
of $\A$-isotopy.
Thus 
if
$(\ep_1,\ep_2)\ne(\ep'_1,\ep'_2)$,
then
$f_{(\ep_1,\ep_2)}^k$ and $f_{(\ep'_1,\ep'_2)}^k$
are not $\A$-isotopic.
We remark that $\ep_1\ep_2$ is known as the local degree of $f$.
The algebraic sum of it is related to the topology of
the source and the target manifolds 
(See \cite{fukuda,quine,saeki} for the $\Z_2$-case, and
\cite{suyak} for the $\Z$-case. See also \cite{df}).

\subsection{The case $k=n=4l+1$ and $l\ne0$}\label{sec:odd1}
Here, we use the following terminology:
Let $I$ be a set of indices such that
$\#I$ is even.
Then the {\em $\pi$-rotations of\/ $I$\/} stands for the diffeomorphism-germ
$(x_1,\ldots,x_{k})\mapsto
(\tilde x_1,\ldots,\tilde x_{k}),$
where $\tilde x_j=\varepsilon x_j$ if $j\in I$,
and $\tilde x_j=x_j$ if $j\not\in I$, with $\varepsilon=-1$.
We see that 
applying 
$\pi$-rotations on the source space and on the target space
does not change the $\mathcal{A}$-isotopy class.

We assume that $\ep_2=-1$. Since $l\ne0$, the number
$
\#\{1,4,6,\ldots,4l\}
$
is even. 
By $\pi$-rotations of
$
\{1,4,6,\ldots,4l\}
$
on the source space, we see that
$f_{(\ep_1,\ep_2)}^{4l+1}$ 
is $\A$-isotopic to
\begin{equation}\label{eq:4l11}
\Big(
\ep_1\big(
x_2x_1+x_3x_1^2+\cdots+x_{k}x_1^{k-1}+x_1^{k+1}
\big),\ 
\ep_2x_2,\ x_3,\ \ep_2x_4,x_5,\ldots,\ \ep_2x_{4l},\ x_{4l+1}\Big).
\end{equation}
Then by $\pi$-rotations of
$
\{2,4,\ldots,4l\}
$ (even number)
on the target space,
we see that $f_{(\ep_1,\ep_2)}^{4l+1}$ is $\A$-isotopic to
$f_{(\ep_1,1)}^{4l+1}$.
On the other hand,
by Lemma \ref{lem:sgninv}, we see that
$(-1)^{4l}\ep_1^{4l+1}\ep_2^{4l}=\ep_1$
is an invariant of $\A$-isotopy.
Thus if $\ep_1\ne\ep'_1$, then
$f_{(\ep_1,1)}^k$ and $f_{(\ep'_1,1)}^k$ are
not $\A$-isotopic.
\subsection{The case $k=n=1$}
We see that the sign of $f_{x_1x_1}^1(0)$ is
an invariant of $\A$-isotopy.
Thus for a given $1$-Morin singularity $f:(\R^1,0)\to(\R^1,0)$,
if $f_{x_1x_1}^1(0)>0$ (respectively, $f_{x_1x_1}^1(0)<0$),
$f$ is $\A$-isotopic to $x_1^2$ (respectively, $-x_1^2$).
\subsection{The case $k=n=4l+2$}
We assume that $\ep_2=-1$. By $\pi$-rotations of
$
\{1,2,\underbrace{3,5,\ldots,4l+1}_{\substack{\text{even}}}\}
$
on the source space, we see that
$f_{(\ep_1,\ep_2)}^k$ is
$\A$-isotopic to
$$
\Big(
\ep_1\ep_2\big(
x_2x_1+x_3x_1^2+\cdots+x_{k}x_1^{k-1}+x_1^{k+1}
\big),\ 
x_2,\ \ep_2x_3,\ x_4,\ \ep_2x_5,\ \ldots,
\ \ep_2x_{4l+1},\ x_{4l+2}\Big).
$$
Then by $\pi$-rotations of
$
\{3,5,\ldots,4l+1\}
$
on the target space, it is $\A$-isotopic to
$f_{(\ep_1\ep_2,1)}^k$.
On the other hand,
by Lemma \ref{lem:sgninv}, we see that
if $\ep_1\ep_2$ is an invariant of the $\A$-isotopy.
Hence $f_{(\ep,1)}^k$ and $f_{(\ep',1)}^k$ are
not $\A$-isotopic if $\ep\ne\ep'$, where $\ep,\ep'\in\{\pm1\}$.
Like as in the case of $k=n=4l$, the invariant
$\ep_1\ep_2$ is known as the local degree of $f$,
and algebraic sum of it is related to the topology of
the source and target manifolds (\cite{df,fukuda,quine,saeki,suyak}).
\subsection{The case $k=n=4l+3$}\label{sec:odd2}
We assume that $\ep_1=-1$. By $\pi$-rotations of
$
\{1,2,4,\ldots,4l+2\}
$ (even number)
on the source space, $f_{(\ep_1,\ep_2)}^k$ is $\A$-isotopic to
\begin{equation}
\begin{array}{l}
\Big(
\ep_1\big(
\ep_2x_2x_1+x_3x_1^2+\cdots+x_{k}x_1^{k-1}+x_1^{k+1}
\big),\\
\hspace{30mm}
\ep_1\ep_2x_2,\ x_3,\ \ep_1x_4,\ x_5,\ \ldots,\ \ep_1x_{4l+2},\ x_{4l+3}\Big).
\end{array}
\end{equation}
Again by $\pi$-rotations of
$
\{1,2,4,\ldots,4l+2\}
$ 
on the target space,
we see that $f_{(\ep_1,\ep_2)}^k$ is $\A$-isotopic to
$f_{(1,\ep_2)}^k$.
On the other hand,
by Lemma \ref{lem:sgninv}, the sign
$\ep_1\ep_2\cdot(-1)^{4l+2}\ep_1^{4l+3}\ep_2^{4l+4}=\ep_2$ is
an invariant of $\A$-isotopy.
Thus if $\ep_2\ne\ep'_2$, then
$f_{(1,\ep_2)}^k$ and $f_{(1,\ep'_2)}^k$ are
not $\A$-isotopic.
This invariant is related to the Vassiliev type invariants
of singularities, since we consider isotopy (see \cite{go}).
In \cite{go}, its global properties are also investigated.
See \cite{bs} for another interpretation.

\subsection{The case $n>k$}
If $n>k$, then by a $\pi$-rotation $\{2,n\}$ on the source space,
and by a $\pi$-rotation $\{1,n\}$ on the target space,
we see that $f_{(\ep_1,\ep_2)}^k$ is $\A$-isotopic to
$f_{(\ep_1\ep_2,1)}^k$.
If $k$ is even, by Lemma \ref{lem:sgninv},
$f_{(\ep,1)}^k$
is $\A$-isotopic to
$f_{(\ep',1)}^k$
if and only if $\ep=\ep'$.
If $k$ is odd and $k=4l+1$ ($l\ne0$) 
(respectively, $k=4l+3$), 
assume that $\ep=-1$.
Then by $\pi$-rotations $\{1,2,4,\ldots,4l,n\}$ 
(respectively, $\{1,2,4,\ldots,4l+2\}$) on the source space,
we see that $f_{(\ep,1)}^k$ is $\A$-isotopic to
$$
\begin{array}{l}
\Big(
\ep\big(
x_2x_1+x_3x_1^2+\cdots+x_{4l+1}x_1^{4l}+x_1^{4l+2}
\big),\\
\hspace{30mm}
\ep x_2,\ x_3,\ \ep x_4,\ \ldots,
\ep x_{4l},\ x_{4l+1},\ \ldots,\ x_{n-1},\ 
\ep x_n\Big)\end{array}
$$
if $k=4l+1$, and
$$
\Big(
\ep\big(
x_2x_1+x_3x_1^2+\cdots+x_{4l+1}x_1^{4l}+x_1^{4l+2}
\big),\ 
\ep x_2,\ x_3,\ \ep x_4,\ \ldots,\ 
\ep x_{4l+2},\ x_{4l+3},\ \ldots,\ x_n\Big)
$$
if $k=4l+3$.
Then we easily see that these germs are 
$\A$-isotopic to $f_{(1,1)}^k$.
Furthermore,
in the case of $n>k=1$,
one can easily see that $f_{\ep_1}^1$
is $\A$-isotopic to $f_{1}^1$.

We remark that in the case of $k=1$ and $n>1$,
$1$-Morin singularities are also called {\em folds}.
Thus 
all folds are $\A$-isotopic to $(x_1^2,x_2,\ldots,x_n)$.
This is a special case of Ando's result
which claims that the homotopy types of 
the set of $r$-jets of folds are $O(n)$ \cite[p.169]{ando2}.

\section{Normal forms and invariants}\label{sec:normalform}
We summarize the normal forms and invariants
for each case.
The case of $k=n$ is shown in Table \ref{tab:matome1}
and
the case of $k<n$ is shown in Table \ref{tab:matome2},
where $\#$ indicates the number of $\A$-isotopy classes.
\begin{table}[htbp]
\centering
\begin{tabular}{|c|c|c|c|c|c|}
\hline
name&$k$&    normal&invariants&$\#$\\
 &&form&&\\
\hline
fold&1&$f_{\ep_1}^1$
&$\eta^2 f$&2\\
\hline
cusp&2&$f_{(\ep_1,1)}^k$
&$\eta^2\lambda=\ep_1$
&2\\
\hline
swallowtail&3&$f_{(1,\ep_2)}^k$
&$\eta^3\lambda\,\det \grad(\lambda,\eta\lambda,\eta^2\lambda)=\ep_2$
&2\\
\hline
butterfly&4&$f_{(\ep_1,\ep_2)}^k$
&\big($\eta^4\lambda,\det \grad(\lambda,\ldots,\eta^3\lambda)\big)$
&4\\
&&&\hspace{30mm}$=(\ep_1\ep_2,-\ep_2)$&\\
\hline
&$5$&$f_{(\ep_1,1)}^k$
&$\det \grad(\lambda,\ldots,\eta^4\lambda)=\ep_1$
&2\\
\hline
\vdots&\vdots&\vdots&\vdots&\vdots\\
\hline
&$4l$&$f_{(\ep_1,\ep_2)}^k$
&\big($\eta^{4l}\lambda,
\det \grad(\lambda,\ldots,\eta^{4l-1}\lambda)\big)$
&4\\
&&&\hspace{30mm}$=(\ep_1\ep_2,-\ep_2)$&\\
\hline
&$4l+1$&$f_{(\ep_1,1)}^k$
&$\det \grad(\lambda,\ldots,\eta^{4l}\lambda)=\ep_1$
&2\\
\hline
&$4l+2$&$f_{(\ep_1,1)}^k$
&$\eta^{4l+2}\lambda=\ep_1$&2\\
\hline
&$4l+3$&$f_{(1,\ep_2)}^k$
&$\eta^{4l+3}\lambda\,\det \grad(\lambda,\ldots,\eta^{4l+2}\lambda)=\ep_2$
&2\\
\hline
\end{tabular}
\caption{$\A$-isotopy classes of $n$-Morin singularities in $C^\infty(n,n)$.}
\label{tab:matome1}
\end{table}

\begin{table}[htbp]
\centering
\begin{tabular}{|c|c|c|c|c|c||c|}
\hline
name&$k$    &\hspace{3mm}normal form\hspace{3mm}
&invariants&$\#$\\
\hline
fold ($\times$ intervals)&1
&$f_{1}^1$&-&1\\
\hline
cusp ($\times$ intervals)&2
&$f_{(\ep_1,1)}^k$&$\eta^2\lambda=\ep_1$
&2\\
\hline
swallowtail ($\times$ intervals)&3&$f_{(1,1)}^k$
&-&1\\
\hline
$\vdots$&$\vdots$&$\vdots$&$\vdots$&$\vdots$\\
\hline
 &$2m$
&$f_{(\ep_1,1)}^k$&$\eta^{2m}\lambda=\ep_1$&2\\
\hline
 &$2m+1$&$f_{(1,1)}^k$&-&1\\
\hline
\end{tabular}
\caption{$\A$-isotopy classes of $k$-Morin singularities in $C^\infty(n,n)$.}
\label{tab:matome2}
\end{table}
\section{Other singularities}\label{sec:other}
In this section, we consider $\A$-isotopy
for other well-known low-dimensional
singularities.
\subsection{Codimension one map-germs from the plane into the plane}
Classification up to $\A$-equivalence 
for map-germs from the plane into the plane 
is given by Rieger \cite{rieger}.
He classified map-germs $(\R^2,0)\to(\R^2,0)$ with corank one
and $\A_e$-codimension $\leq 6$.
Table \ref{tab:rie}
shows the list of the $\A_e$-codimension $\leq 1$
local singularities obtained in \cite{rieger}.
\begin{table}[!htbp]
\centering
\begin{tabular}{|c|c|c|}
\hline
name&normal form&$\A_e$-codimension\\
\hline
fold&$(x_1^2,x_2)$&0\\
\hline
cusp&$(x_1^3+x_1x_2,x_2)$&0\\
\hline
lips&$(x_1^3+ x_1x_2^2,x_2)$&1\\
\hline
beaks&$(x_1^3- x_1x_2^2,x_2)$&1\\
\hline
(planar) swallowtail&$(x_1^4+x_1x_2,x_2)$&1\\
\hline
\end{tabular}
\caption{Classification of $C^\infty(2,2)$}
\label{tab:rie}
\end{table}
Folds and cusps are Morin singularities.
Recognition criteria for other singularities
are given in \cite{sajihiro}:
\begin{lemma}
Let\/ $f\in C^\infty(2,2)$ be a map-germ.
\begin{enumerate}
\item[{\rm (1)}] $f$ is\/ $\A$-equivalent to a lips if
and only if\/
$d\lambda=0$, $\det\hess\lambda>0$ at\/ $0$.
\item[{\rm (2)}] $f$ is\/ $\A$-equivalent to a beaks if
and only if\/
$d\lambda=0$, $\det\hess\lambda<0$,
$\eta\eta\lambda\ne0$ at\/ $0$.
\item[{\rm (3)}] $f$ is\/ $\A$-equivalent to a\/ $($planar$)$ swallowtail if
and only if\/
$d\lambda\ne0$, $\eta\lambda=\eta\eta\lambda=0$,
$\eta\eta\eta\lambda\ne0$ at\/ $0$.
\end{enumerate}
\end{lemma}
We have the following theorem.
\begin{theorem}
Let\/ $f\in C^\infty(2,2)$ be a map-germ.
\begin{enumerate}
\item If\/ $f$ is\/ $\A$-equivalent to a lips,
and\/ $\sign\eta\eta\lambda=\ep$,
then\/ $f$ is\/ $\A$-isotopic to\/
$(\ep x_1(x_1^2+x_2^2),x_2)$.
Moreover these two map-germs are not\/ $\A$-isotopic.
\item If\/ $f$ is\/ $\A$-equivalent to a beaks,
and\/ $\sign\eta\eta\lambda=\ep$,
then\/ $f$ is\/ $\A$-isotopic to\/
$(\ep x_1(x_1^2-x_2^2),x_2)$.
Moreover these two map-germs are not\/ $\A$-isotopic.
\item If\/ $f$ is\/ $\A$-equivalent to a\/ $($planar$)$ swallowtail,
and\/ $\sign(\xi\lambda\,\eta\eta\eta\lambda)=\ep$,
then\/ $f$ is\/ $\A$-isotopic to\/
$(\ep x_1x_2+x_1^4,x_2)$,
where\/ $\xi$ is a vector field such that\/
$(\xi,\eta)$ is a positive frame at\/ $0$.
Moreover these two map-germs are not\/ $\A$-isotopic.
\end{enumerate}
Here, $\ep=\pm1$.
\end{theorem}
\begin{proof}
By the same method as in the proof of Lemma \ref{lem:prenormal},
we may assume $f$ has the form $f(x_1,x_2)=(f_1(x_1,x_2),x_2)$.
There exist functions $g_1(x_1,x_2)$ and
$g_2(x_2)$ such that
$f_1(x_1,x_2)=x_1g_1(x_1,x_2)+g_2(x_2)$.
Thus we may assume that
$f(x_1,x_2)=(x_1g_1(x_1,x_2),x_2)$ in all cases.\\
({\em Proofs of\/ {\rm (1)} and\/ {\rm (2)}}.)
Since the function $\lambda$ satisfies
that $\lambda(0)=0$ and $d\lambda(0)=0$,
$f$ can be written as
$$
\big(ax_1^2+bx_1x_2+cx_2^2+h(x_1,x_2),x_2\big)\quad(a,b,c\in\R),
$$
where $h(x_1,x_2)$ is a function which
order is greater than $3$.
Since $\eta\eta\lambda\ne0$, it holds that $c\ne0$.
Thus by an orientation-preserving diffeomorphism-germ
$\tilde x_1=x_1,$ $\tilde x_2=x_2+2bx_1/3c$ $(t\in[0,1])$,
and by a suitable scaling change,
$f$ is $\A$-isotopic to the map-germ
$(x_1,x_2^2,\pm x_2(x_1^2\pm x_2^2)+h(x_1,x_2))$,
where $h(x_1,x_2)$ is a function whose
order is greater than $3$.
It is well known that lips and beaks are
three-determined (\cite[Lemma 3.1.3]{rieger}),
and the proof of it contains that
$(x_1,x_2^2,\pm x_2(x_1^2\pm x_2^2))$
and
$(x_1,x_2^2,\pm x_2(x_1^2\pm x_2^2)+h(x_1,x_2))$
are $\A$-isotopic 
(see \cite[Section 3]{ondet}, see also \cite[Section 3]{mather3}).
Since $\eta\eta\lambda$ does not change by positive
coordinate changes on the source and target,
the second assertion of the theorem is obvious.\\
({\em Proof of\/ {\rm (3)}}.)
We can write 
$$
\begin{array}{l}
x_1g_1(x_1,x_2)
=
a_{20}x_1^2+a_{11}x_1x_2
+a_{30}x_1^3+a_{21}x_1^2x_2+a_{12}x_1x_2^2\\
\hspace{23mm}
+a_{40}x_1^4+a_{31}x_1^3x_2+a_{22}x_1^2x_2^2+a_{13}x_1x_2^3
+h(x_1,x_2)
\quad
(a_{**}\in\R),
\end{array}
$$
where $h(x_1,x_2)$ is a function whose
order is greater than $4$.
Since $\eta\lambda(0)=\eta\eta\lambda(0)=0$ and
$d\lambda(0)\ne0$, it holds that
$a_{20}=a_{30}=0$ and $a_{11}\ne0$,
thus by an orientation-preserving diffeomorphism-germ
$\tilde x_1=|a_{11}|x_1
+a_{21}x_1^2+a_{12}x_1x_2
a_{31}x_1^3+a_{22}x_1^2x_2+a_{13}x_1x_2^2
$
and by a suitable scaling change,
$f$ is $\A$-isotopic to the map-germ
$(x_1,\pm x_1x_2+x_2^4+h(x_1,x_2))$.
By the same argument as just above,
we see that $f$ is $\A$-isotopic to
$(x_1,\pm x_1x_2+x_2^4)$.
Since the sign of the product 
$\xi\lambda\eta\eta\eta\lambda$ 
does not depend on the choice of $(\xi,\eta)$,
the second assertion is obvious.
\end{proof}
\subsection{Whitney umbrellas and $S_1$-singularities}
Classification for map-germs from the plane into the $3$-space 
up to $\A$-equivalence is given by Mond \cite{mond}.
He classified simple map-germs $(\R^2,0)\to(\R^3,0)$.
Table \ref{tab:mond}
shows the list of the $\A_e$-codimension $\leq 1$
local singularities obtained in \cite{mond}.
\begin{table}[!htbp]
\centering
\begin{tabular}{|c|c|c|}
\hline
name&normal form&$\A_e$-codimension\\
\hline
Whitney umbrella&$(x_1^2,x_1x_2,x_2)$&0\\
\hline
$S_1^+$&$(x_1^2,x_1(x_1^2+ x_2^2),x_2)$&1\\
\hline
$S_1^-$&$(x_1^2,x_1(x_1^2- x_2^2),x_2)$&1\\
\hline
\end{tabular}
\caption{Classification of $C^\infty(2,3)$}
\label{tab:mond}
\end{table}
In the list, $S_1^\pm$ singularities are 
also called {\em Chen-Matumoto-Mond $\pm$-singularities\/} (\cite{cm}).
Recognition criteria for them
are given in \cite{sajisk}.
Let $f\in C^\infty(2,3)$ be a corank one map-germ at $0$ and
$\eta$ a non-zero vector field such that $\eta(0)\in \ker (df)(0)$.
Let $\xi$ be a vector field such that $\xi,\eta$ are
linearly independent.
We set
\begin{equation}\label{eq:s1hantei}
w=\det(\xi f,\ \eta f,\ \eta\eta f).
\end{equation}
Then $f$ is a Whitney umbrella if and only if $dw\ne0$ at $0$.
Furthermore, $f$ is an $S_1^+$ singularity
(respectively, $S_1^-$ singularity)
if and only if
$dw=0$ and $\det\hess w(0)>0$
(respectively, 
$dw=0$, $\det\hess w(0)<0$ and $\eta\eta w(0)\ne0$)
\cite[Theorem 2.2]{sajisk}.
For $\A$-isotopy, we have the following theorem:
\begin{theorem}
Let\/ $f\in C^\infty(2,3)$ be a corank one map-germ at\/ $0$.
\begin{enumerate}
\item[$(1)$] If\/ $f$ is\/ $\A$-equivalent to a Whitney umbrella
then\/ $f$ is\/ $\A$-isotopic to\/ $(x_1^2,x_1x_2,x_2)$.
\item[$(2)$] If\/ $f$ is\/ $\A$-equivalent to a\/ 
$S_1^+$ singularity,
and\/ $\sign\eta\eta w=\ep$,
then\/ $f$ is\/ $\A$-isotopic to\/
$\big(x_1^2,\ep x_1(x_1^2+x_2^2),x_2\big)$.
Moreover these two map-germs are not\/ $\A$-isotopic.
\item[$(3)$] If\/ $f$ is\/ $\A$-equivalent to a\/
$S_1^-$ singularity,
and\/ $\sign\eta\eta w=\ep$,
then\/ $f$ is\/ $\A$-isotopic to\/
$\big(x_1^2,\ep x_1(x_1^2-x_2^2),x_2\big)$.
Moreover these two map-germs are not\/ $\A$-isotopic.
\end{enumerate}
Here, $\ep=\pm1$.
\end{theorem}
\begin{proof}
Let $f\in C^\infty(2,3)$ be a corank one map-germ at $0$.
Then one can easily see that $f$ is $\A$-isotopic to
the map-germ of the form
$(x_1^2,x_1h(x_1^2,x_2),x_2)$
for some function $h$ satisfying 
$h(0)=0$ (see \cite[p72]{sajisk}, for example).
We may choose $\xi=\partial x_2$ and 
$\eta=\partial x_1$.
Then the function $w$ defined in \eqref{eq:s1hantei}
is
$$
w(x_1,x_2)=
8x_1^2h_{x_1}(x_1^2,x_2)+8x_1^4h_{x_1x_1}(x_1^2,x_2)
-2h(x_1^2,x_2).
$$
Thus we have
$$
\xi w(0)=-2h_{x_2}(0),\ 
\det\hess w(0)=-24h_{x_2x_2}(0)h_{x_1}(0),\ 
\eta\eta w=12h_{x_1}(0).
$$
Hence  (1) is obvious.
We prove $(2)$ and $(3)$.
We assume that $dw_{(0)}=0$,
$\det\hess w(0)$ $\ne0$ and $\eta\eta w(0)\ne0$.
Since $h(0)=h_{x_2}(0)=0$,
$h_{x_2x_2}(0)\ne0$ and $h_{x_1}(0)\ne0$,
there exist functions $\bar h$ and $\tilde h$ 
satisfying $\bar h(0)=\tilde h(0)=0$ such that
$$
h(x_1,x_2)=
\alpha x_1^2\big(1+\tilde h(x_1^2,x_2)\big)
+
\beta x_2^2\big(1+\bar h(x_2)\big).
$$
We remark that $\det\hess w(0)=-48\alpha\beta$.
Thus by a coordinate change 
$$
(x_1,x_2)
\mapsto
\left(x_1\sqrt{1+\tilde h(x_2)},\ 
x_2\sqrt{1+\bar h(x_1^2,x_2)}\right)
$$
on the source, and a suitable scale change,
we see the first assertions of (2) and (3).
The second assertions are obvious.
\end{proof}

\section{Perturbation of simple singularities}\label{sec:perturb}
Consider a simple corank $1$ 
singularity $f\in C^\infty(n,n)$
and a small stable perturbation $\tilde f$ of $f$.
Since all stable corank $1$ singularities are Morin 
singularities, $\tilde f$ has some $n$-Morin singularities.
In the complex case, the number of
$n$-Morin singularities appearing in $\tilde f$
is constant, but in the real case,
it is not constant and
the maximal number has been studied \cite{fi,fns}.
It is denoted by $c(f)$,
and it represents a geometric property of $f$.
In the present paper, we divide $\A$-classes into
$\A$-isotopy classes, and
we can study $c(f)$ more precisely using $\A$-isotopy.
In this section, we observe the $\A$-isotopy classes 
of $n$-Morin singularities appearing on some perturbations of
simple singularities.
Since the numbers of $\A$-isotopy classes 
of $n$-Morin singularities has a periodicity $4$ with respect to $n$,
we consider the cases of $n=2,3,4$ and $5$ here.
For the sake of simplicity, in the case of $n=4$,
we call 
the invariant $\eta^4\lambda$ the {\em first invariant},
and
the invariant $\det \grad(\lambda,\ldots,\eta^3\lambda)$ 
the {\em second invariant}.

\subsection{Classification of simple corank $1$ map-germs of $C^\infty(n,n)$}
Let $f$ be a corank $1$ map-germ of $C^\infty(n,n)$.
Then $f$ is $\A$-equivalent to the map-germ 
$$
(t,x_2,\ldots,x_{n})\mapsto
\big(f_1(t,x_2,\ldots,x_{n}),x_2,\ldots,x_{n}\big).
$$
The function $f_1(t,0,\ldots,0)$ is called
the {\em genotype\/} of $f$.
If $f$ is simple and $n\geq3$, then the genotype of $f$
is $t^{i+1}$ ($i\leq n+1$).
If $f$ is simple and $n=2$, then the genotype of $f$
is $t^{i+1}$ ($i\leq 4$).
Thus, if $f$ is simple, one can show that $f$ is 
$\A$-equivalent to
\begin{equation}\label{eq:normalform}
(t,x_2,\ldots,x_{n})\mapsto
\left(
t^{i+1}+\sum_{j=1}^{i-1}p_j(x_2,\ldots,x_{n})t^j,\ 
x_2,\ \ldots,\ x_{n}
\right).
\end{equation}
We denote the map-germ of the form \eqref{eq:normalform}
by $[p_1,\ldots,p_{i-1}]$.

Classification of simple map-germs of $C^\infty(2,2)$
is given by Rieger \cite{rieger}, and 
classification of simple map-germs of $C^\infty(3,3)$
is given by Marar and Tari \cite{mt}
using the method of complete transversal (\cite{bkdp}).
Using the same method as in \cite{mt}, we can find three families of 
simple map-germs as in Table \ref{tab:germs}.
It should be remarked that
these families can be also obtained
by the {\em augmentation\/} of map-germs \cite{cmw,h1,h2},
because $[x_2,\ldots,x_{n}]$ is a stable map-germ.
\begin{table}[htbp]
\centering
\label{tab:germs}
\begin{tabular}{|c|c|c|c|c|c|}
\hline
 &$n$&2&3&4&5\\
\hline
\hline
Family&genotype&$t^3$&$t^4$&$t^5$&$t^6$\\
A &map-germ&$[x_2^l]$&$[x_2,x_3^l]$&$[x_2,x_3,x_4^l]$&
$[x_2,x_3,x_4,x_5^l]$\\
\hline
Family&genotype&$t^4$&$t^5$&$t^6$&$t^7$\\
B&map-germ&$[x_2,0]$&$[x_2,x_3,0]$&$[x_2,x_3,x_4,0]$&
$[x_2,x_3,x_4,x_5,0]$\\
\hline
Family&genotype&$t^4$&$t^5$&$t^6$&$t^7$\\
C&map-germ&$[x_2^2,x_2]$&$[x_2,x_3^2,x_3]$&$[x_2,x_3,x_4^2,x_4]$&
$[x_2,x_3,x_4,x_5^2,x_5]$\\
\hline
\end{tabular}
\caption{Families of map-germs ($l\geq2$).}
\end{table}

Observing the invariants
of $n$-Morin singularities $\tilde f$ for the
families A, B and C,
we may clarify the difference between
these families with respect to $n$.

\subsection{Family A}
A versal unfolding of a map-germ in family A
is
$F_u(t,x)
=
\big(q(t,x,u),x_2,\ldots,x_{n}\big),$
$$
\begin{array}{rcl}
q(t,x,u)
&=&
\left\{
\begin{array}{lll}
t^3+\bar q(x_2,u)t,
&(n=2),\\
t^4+x_2t+\bar q(x_3,u)t^2,
&(n=3),\\
t^5+x_2t+x_3t^2+\bar q(x_4,u)t^3,
&(n=4),\\
t^6+x_2t+x_3t^2+x_4t^3+\bar q(x_5,u)t^4,&(n=5),
\end{array}
\right.\\[9mm]
\bar q(x_n,u)&=&x_n^l+u_0+u_1x_n+\cdots+u_{l-2}x_n^{l-2},
\end{array}
$$
where 
$x=(x_2,\ldots,x_{n})$ and $u=(u_0,\ldots,u_{l-2})\in\R^{l-1}$.
For versal unfolding, see \cite[Chapter XIV]{martinet}, for example.
Since $\lambda=q_t(t,x,u)$ and $\eta=\partial t$,
we have Table \ref{tab:fama},
where ``$n$-M'' means the condition for $n$-Morin singularity,
and ``inv'' means the value of the invariants mentioned in 
Lemma \ref{lem:sgninv}.
Looking at Table \ref{tab:fama},
we observe
in the case $n=2$, where
all $n$-Morin singularities are $\A$-isotopic,
and in
the cases $n=3$ and $5$,
there are two kinds of $n$-Morin singularities,
and in the case $n=4$,
there are two kinds of $n$-Morin singularities,
but in all cases, the first and second invariants coincide.
\begin{table}[!htbp]
\centering
\label{tab:fama}
\begin{tabular}{|c|c|c|c|c|}
\hline
$n$&\hspace{8mm}2\hspace*{8mm}&\hspace{8mm}3\hspace*{8mm}&
\hspace{8mm}4\hspace*{8mm}&\hspace{8mm}5\hspace*{8mm}\\
\hline
$c(f)$ &$l$&$l$&$l$&$l$\\
\hline
$n$-M &$t=0$, $\bar q=0$, $\bar q_{x_2}\ne0$&
\multicolumn{3}{|c|}{$t=x_2=\cdots =x_{n-1}=0,\ \bar q=0,\ 
\bar q_{x_{n}}\ne0$}\\
\hline
inv&1&$\bar q_{x_3}$&$(1,\bar q_{x_4})$&$\bar q_{x_5}$\\
\hline
\end{tabular}
\caption{Maximal number of Morin singularities and their invariants
(family A)}
\end{table}
\subsection{Family B}
A versal unfolding of a map-germ in family B
is
$F_u(t,x)
=
\big(q(t,x,u_0),x_2,\ldots,x_{n}\big)
$,
$$
q(t,x,u_0)
=
\left\{
\begin{array}{ll}
t^4+x_2t+u_0t^2&(n=2),\\
t^5+x_2t+x_3t^2+u_0t^3&(n=3),\\
t^6+x_2t+x_3t^2+x_4t^3+u_0t^4&(n=4),\\
t^7+x_2t+x_3t^2+x_4t^3+x_5t^4+u_0t^5&(n=5),
\end{array}
\right.
$$
where $u_0\in\R$.
Since $\lambda=q_t(t,x,u_0)$ and $\eta=\partial t$,
we have Table \ref{tab:famb}.
Looking at Table \ref{tab:famb},
we observe that in
the case $n=2$,
all $n$-Morin singularities are $\A$-isotopic,
and in
the cases $n=3$ and $5$,
there are two kinds of $n$-Morin singularities,
and in the case $n=4$,
there are two kinds of $n$-Morin singularities,
but in all cases, the first and second invariants coincide.
\begin{table}[!htbp]
\centering
\label{tab:famb}
\begin{tabular}{|c|c|c|c|c|}
\hline
$n$&2&3&4&5\\
\hline
$c(f)$ &$2$&$2$&$2$&$2$\\
\hline
$n$-M
&
{\small $\begin{array}{l}t=\pm\sqrt{-\dfrac{u_0}{6}},\\
                 x_2=8t^3,\\
                 t\ne0
\end{array}$}
&
{\small $\begin{array}{l}t=\pm\sqrt{-\dfrac{u_0}{10}},\\
                 x_2=105t^4,\\
                 x_3=-40t^3,\\
                 t\ne0
\end{array}$}
&
{\small $\begin{array}{l}t=\pm\sqrt{-\dfrac{u_0}{15}},\\
                 x_2=24t^5,\\
                 x_3=-45t^4,\\
                 x_4=40t^3,\\
                 t\ne0
\end{array}$}
&
{\small $\begin{array}{l}t=\pm\sqrt{-u_0/21},\\
                 x_2=-35t^6,\\
                 x_3=84t^5,\\
                 x_4=-105t^4,\\
                 x_5=70t^3,\\
                 t\ne0
\end{array}$}\\
\hline
inv&$t$&$t^2$&$(t,t)$&$t$\\
\hline
\end{tabular}
\caption{Maximal number of Morin singularities and their invariants
(family B)}
\end{table}

\subsection{Family C}
A versal unfolding of a map-germ in family C
is
$F_u(t,x)
=
\big(q(t,x,u),x_2,\ldots,x_{n}\big)
$,
$$
q(t,x,u)
=
\left\{
\begin{array}{ll}
t^4+(x_2^2+u_0+u_1x_2)t+x_2t^2&(n=2),\\
t^5+x_2t+(x_3^2+u_0+u_1x_3)t^2+x_3t^3&(n=3),\\
t^6+x_2t+x_3t^2+(x_4^2+u_0+u_1x_4)t^3+x_4t^4&(n=4),\\
t^7+x_2t+x_3t^2+x_4t^3+(x_5^2+u_0+u_1x_5)t^4+x_5t^5&(n=5),
\end{array}
\right.
$$
where $u=(u_0,u_1)\in\R^2$.
Since $\lambda=q_t(t,x,u_0,u_1)$ and $\eta=\partial t$,
we have Table \ref{tab:famc}.
Here, 
equations $C_2$ stands for the equations
$36t^4-8t^3-6u_1t^2+u_0=0, x_2=-6t^2,t\ne0$,
equations $C_3$ stands for the equations
$100t^4-20t^3-10u_1t^2+u_0=0,x_2=25t^4-200t^5+20u_1t^3-2u_0t,
x_3=-10t^2, t\ne0$,
equations $C_4$ stands for the equations
$255t^4-40t^3-15t^2u_1+u_0=0,
x_2=3(225t^6-32t^5-15t^4u_1+u_0t^2),
x_3=-3(225t^5-25t^4-15t^3u_1+tu_0),
x_4=-15t^2,
t\ne0$
and
equations $C_5$ stands for the equations
$441t^4-70t^3-21t^2u_1+u_0=0,
x_2=-4u_0t^3+84u_1t^5+245t^6-1764t^7,
x_3=2640t^6-336t^5-126u_1t^4+6u_0t^2,
x_4=-1764t^5+175t^4+84u_1t^3-4u_0t,
x_5=-21t^2,
t\ne0$.
%Looking at Table \ref{tab:famc},
%we observe that the invariants of the
%four $n$-Morin singularities 
%can take both positive and negative values,
%different from the family B.
%For example, 
%if $n=3$, and $(u_0,u_1)=(1/100,1)$
%then the sign of the invariants 
%of the four $3$-Morin singularities appearing
%on $F_u(x)$ are $(-,-,+,+)$.

\begin{table}[!htbp]
\centering
\label{tab:famc}
\begin{tabular}{|c|c|c|c|c|}
\hline
$n$&2&3&4&5\\
\hline
$c(f)$ &$4$&$4$&$4$&$4$\\
\hline
$n$-M
&equations
&equations
&equations
&equations\\
&$C_2$
&$C_3$
&$C_4$
&$C_5$\\
\hline
inv&$t$&$-20 t^2+3 t+u_1$&$\big(t,t(30t^2-4t-u_1)\big)$
&$t(-42t^2+5t+u_1)$\\
\hline
\end{tabular}
\caption{Maximal number of Morin singularities and their invariants
(family C)}
\end{table}
\section{Criteria for $\Sigma^{2,0}$ singularities and isotopy}

In this section, we consider a corank two singularity
for $C^\infty(4,4)$.
$1$-Morin, $2$-Morin, $3$-Morin singularities and
$\Sigma^{2,0}$-singularities are stable (equivalently, gene\-ric) singularities for
maps from $4$-manifolds to $4$-manifolds.
Let $f\in C^\infty(4,4)$ be a 
stable map-germ such that the origin is a singular point of $f$.
Then $f$ is $\A$-equivalent to
$1$-Morin, 
$2$-Morin, 
$3$-Morin singularity or the following map-germ:
$$
\begin{array}{rl}
\Sigma^{2,0}_{{\rm hyp}}&:(x_1,x_2,x_3,x_4)\mapsto
(x_1^2+x_2x_3,x_2^2+x_1x_4,x_3,x_4)\\[2mm]
\Sigma^{2,0}_{{\rm elli}}&:(x_1,x_2,x_3,x_4)\mapsto
(x_1^2-x_2^2+x_1x_3+x_2x_4,x_1x_2+x_1x_4-x_2x_3,
x_3,x_4).
\end{array}
$$
The germ $\Sigma^{2,0}_{{\rm hyp}}$ 
(respectively $\Sigma^{2,0}_{{\rm elli}}$)
is also called
the {\em hyperbolic umbilic}
(respectively, the {\em elliptic umbilic}),
and $I_{2,2}^+$ (respectively, $I_{2,2}^-$) \cite{GG,mather6}.
Moreover, we define ``signed'' umbilics as follows:
$$
\begin{array}{rl}
\Sigma^{2,0}_{{\rm hyp},\ep_1}&:(x_1,x_2,x_3,x_4)\mapsto
(x_1^2+x_2x_3,x_2^2+\ep_1x_1x_4,x_3,x_4)\\[2mm]
\Sigma^{2,0}_{{\rm elli},\ep_1,\ep_2}&:(x_1,x_2,x_3,x_4)\mapsto
(x_1^2-x_2^2+\ep_1x_1x_3+x_2x_4,\\
&\hspace{40mm}\ep_1x_1x_2+\ep_1x_1x_4-x_2x_3,
x_3,\ep_2x_4),
\end{array}
$$
where $\ep_1=\pm1$ and $\ep_2=\pm1$.
Then we have the following theorem.
\begin{theorem}\label{thm:sigma20cris}
Let\/ $f\in C^\infty(4,4)$ be a map-germ such that\/
$\rank (df)(0)=2$ holds.
Then\/ $f$ is\/ $\A$-isotopic to\/ $\Sigma^{2,0}_{{\rm hyp},\ep_1}$
$($respectively, $\Sigma^{2,0}_{{\rm elli},\ep_1,\ep_2})$
if and only if
for a coordinate system\/ $(X_1,X_2,X_3,X_4)$ on the target
satisfying that\/
$d(X_1\circ f)(0)=d(X_2\circ f)(0)={0}$
and a pair of vector fields\/ $(\xi,\eta)$ on the source satisfying that\/
$\spann{\xi(0),\eta(0)}=\ker (df)(0)$,
it holds that
\begin{itemize}
\item[$(1)$] $\det\hess_{(\xi,\eta)}\lambda(\zv)<0$ 
$($respectively, $\det\hess_{(\xi,\eta)}\lambda(\zv)>0)$,
\item[$(2)$] $\sign\det\big(\grad(\xi f_1),\ \grad(\xi f_2),\ 
\grad(\eta f_1),\ \grad(\eta f_2)\big)=-\ep_1$
holds\\
$($respectively, 
$\sign\det\big(\grad(\xi f_1),\ \grad(\xi f_2),\ 
\grad(\eta f_1),\ \grad(\eta f_2)\big)=\ep_1$ holds,
and the sign of\/ $\trace\hess_{(\xi,\eta)}\lambda(\zv)$ 
is equal to\/ $\ep_1\ep_2)$.
\end{itemize}
\end{theorem}
Here, 
$f_i=X_i\circ f$,
and $\xi f_i$ denotes the 
directional derivative of $f_i$
with respect to $\xi$,
and
$\lambda$ is the determinant of the Jacobi matrix of $f$
and $\hess_{(\xi,\eta)}\lambda$ is the Hessian matrix of $\lambda$
with respect to $\xi$ and $\eta$:
$$
\hess_{(\xi,\eta)}\lambda
=
\pmt{
\xi\xi\lambda&\xi\eta\lambda\\
\eta\xi\lambda&\eta\eta\lambda}.
$$
We remark that since $\xi(0)$ and $\eta(0)$ 
belong to the kernel of $(df)(0)$, it holds that
$\hess_{(\xi,\eta)}\lambda$ is a symmetric matrix.
Since $\rank (df)(0)=2$, it holds 
that $\lambda$ has a critical point at $0$.
The proof of this theorem is given as follows.

\begin{lemma}\label{lem:sigma201}
Conditions\/ $(1)$ and\/ $(2)$ of Theorem\/ $\ref{thm:sigma20cris}$
do not depend on the choice of vector fields spanning
$\ker (df)(0)$ at $0$.
\end{lemma}
\begin{proof}
Let $\xi,\eta,\zeta,\omega$ be a quadruple of vector fields
of $(\R^4,0)$ such that $\xi,\eta,\zeta,\omega$
are linearly independent.
Let a function $\lambda:(\R^4,0)\to(\R,0)$ has 
a critical point at $0$.
We remark that $\xi\eta\lambda=\eta\xi\lambda$ holds
at $0$.
Let
$\bar\xi,\bar\eta,\bar\zeta,\bar\omega$ be another quadruple of
vector fields such that
$$
\trans{\pmt{\bar\xi,\bar\eta,\bar\zeta,\bar\omega}}
=
%\pmt{a_{11}&a_{12}&a_{13}&a_{14}\\
%     a_{21}&a_{22}&a_{23}&a_{24}\\
%     a_{31}&a_{32}&a_{33}&a_{34}\\
%     a_{41}&a_{42}&a_{43}&a_{44}}
A\ 
\trans{\pmt{\xi, \eta, \zeta, \omega}},
$$
where $A=(a_{ij})_{i,j=1,\ldots,4}$
and
$a_{13}=a_{14}=a_{23}=a_{24}=0$ hold at $0$.
Here, $\trans{(~)}$ means the transpose matrix.
Then it holds that
$$
\pmt{\bar\xi\bar\xi\lambda &\bar\xi\bar\eta\lambda\\
     \bar\eta\bar\xi\lambda&\bar\eta\bar\eta\lambda}
=
A
%\pmt{a_{11}&a_{12}&a_{13}&a_{14}\\
%     a_{21}&a_{22}&a_{23}&a_{24}\\
%     a_{31}&a_{32}&a_{33}&a_{34}\\
%     a_{41}&a_{42}&a_{43}&a_{44}}
\pmt{\xi\xi\lambda &\xi\eta\lambda\\
     \eta\xi\lambda&\eta\eta\lambda}
\trans{
A
%\pmt{a_{11}&a_{12}&a_{13}&a_{14}\\
%     a_{21}&a_{22}&a_{23}&a_{24}\\
%     a_{31}&a_{32}&a_{33}&a_{34}\\
%     a_{41}&a_{42}&a_{43}&a_{44}}
}
$$
at $0$. Thus the independency of the conditions for 
Hessian matrix is proven. Moreover,
$$
\begin{array}{rcl}
&&\big(
a_{11}d(\xi f_1)+a_{12}d(\eta f_1),\ 
a_{11}d(\xi f_2)+a_{12}d(\eta f_2),\\
&&\hspace{40mm}
a_{21}d(\xi f_1)+a_{22}d(\eta f_1),\ 
a_{21}d(\xi f_2)+a_{22}d(\eta f_2)
\big)\\
&=&
\big(d(\xi f_1),\ d(\xi f_2),\ d(\eta f_1),\ d(\eta f_2)\big)
\pmt{a_{11}&0&a_{21}&0\\
     0&a_{11}&0&a_{21}\\
     a_{12}&0&a_{22}&0\\
     0&a_{12}&0&a_{22}}
\end{array}
$$
holds at $0$, where $dh$ means $\grad h$, for the sake of simplicity.
We have
$$
\begin{array}{rcl}
&&\det \big(
a_{11}(d\xi f_1)+a_{12}d(\eta f_1),\ 
a_{11}(d\xi f_2)+a_{12}d(\eta f_2),\\
&&\hspace{30mm}
a_{21}(d\xi f_1)+a_{22}d(\eta f_1),\ 
a_{21}(d\xi f_2)+a_{22}d(\eta f_2)
\big)\\
&=&
(a_{11}a_{22}-a_{12}a_{21})^2
\det\big(d(\xi f_1),\ d(\xi f_2),\ d(\eta f_1),\ d(\eta f_2)\big).
\end{array}
$$
Thus the conditions (1) and (2) do not depend on
the choice of vector fields.
\end{proof}
\begin{lemma}\label{lem:sigma202}
Conditions\/ $(1)$ and\/ $(2)$ of Theorem\/ $\ref{thm:sigma20cris}$
do not depend on the choice of the coordinate system of the target.
\end{lemma}
\begin{proof}
Let $f(x)=(f_1,f_2,f_3,f_4)(x)\in C^\infty(4,4)$ $(x=(x_1,x_2,x_3,x_4))$
be a map-germ
such that
$\rank (df)(0)=2$,
$f_{x_1} =f_{x_2}=0$
and
$df_1=df_2=\zv$ holds at $0$.
Let
$\Phi=(\Phi_1,\Phi_2,\Phi_3,\Phi_4)$ be a diffeomorphism-germ
of $(\R^4,0)$ such that
$d(\Phi_1\circ f)_{\zv}=d(\Phi_2\circ f)_{\zv}=0$.
We show that the condition is the same for
both $f$ and $\Phi\circ f$.
Since
$
(\Phi\circ f)_{x_1}
=
(\Phi\circ f)_{x_2}
=
0
$
holds at $0$, $\ker df$ is the same for $f$ and $\Phi\circ f$.
Since the difference of the determinant of the Jacobi matrix 
is a positive function between $f$ and $\Phi\circ f$,
thus we see the independence of conditions for 
$\hess_{(\xi,\eta)}\lambda$.
Hence it is enough to show that if
$d\big((f_1)_{x_1}\big),d\big((f_1)_{x_2}\big),
d\big((f_2)_{x_1}\big),d\big((f_2)_{x_2}\big)$
is a positive frame, then
$d\big((\Phi_1\circ f)_{x_1}\big),d\big((\Phi_1\circ f)_{x_2}\big),
d\big((\Phi_2\circ f)_{x_1}\big),d\big((\Phi_2\circ f)_{x_2}\big)$
is a positive frame.

Firstly we detect the condition for
$d(\Phi_1\circ f)_{\zv}=d(\Phi_2\circ f)_{\zv}=0$.
By $d(f_1)(0)=d(f_2)(0)=0$, it holds that
$\big((f_i)_{x_1},(f_i)_{x_2},(f_i)_{x_3},(f_i)_{x_4}\big)(0)=\zv$ 
$(i=1,2)$.
Since
$$
d(\Phi_i\circ f)(0)
=\Big((\Phi_i\circ f)_{x_1},\ (\Phi_i\circ f)_{x_2},\ 
(\Phi_i\circ f)_{x_3},\ (\Phi_i\circ f)_{x_4}\Big)(0)
=\zv\qquad(i=1,2),
$$
we have
\begin{equation}\label{eq:sigma203}
%\begin{array}{l}
(\Phi_i)_{x_1}(f_1)_{x_j}+(\Phi_i)_{x_2}(f_2)_{x_j}
+
(\Phi_i)_{x_3}(f_3)_{x_j}+(\Phi_i)_{x_4}(f_4)_{x_j}
=0
%(\Phi_i)_{x_1}(f_1)_{x_4}+(\Phi_i)_{x_2}(f_2)_{x_4}
%+
%(\Phi_i)_{x_3}(f_3)_{x_4}+(\Phi_i)_{x_4}(f_4)_{x_4}
%=0\\
%\end{array}
\qquad(i=1,2,\ j=3,4)
\end{equation}
at $0$. Substituting 
$\big((f_i)_{x_1},(f_i)_{x_2},(f_i)_{x_3},(f_i)_{x_4}\big)(0)=\zv$ $(i=1,2)$
into \eqref{eq:sigma203},
we have
$$
(\Phi_i)_{x_3}(f_3)_{x_3}+(\Phi_i)_{x_4}(f_4)_{x_3}
=0\quad  \text{and}\quad 
(\Phi_i)_{x_3}(f_3)_{x_4}+(\Phi_i)_{x_4}(f_4)_{x_4}
=0
\quad(i=1,2)$$
at $0$. Since
$\rank (df)(0)=2$, it holds that
$\big((f_i)_{x_1},(f_i)_{x_2},(f_i)_{x_3},(f_i)_{x_4}\big)=\zv$ $(i=1,2)$
and
$f_{x_1} =f_{x_2}=0$
at $0$,
we see that 
$\big((f_3)_{x_3},(f_4)_{x_3}\big)(0)$ and 
$\big((f_3)_{x_4},(f_4)_{x_4}\big)(0)$
are linearly independent.
Thus
$
(\Phi_1)_{x_3}=(\Phi_1)_{x_4}=(\Phi_2)_{x_3}=(\Phi_2)_{x_4}=0
$
at $0$.
On the other hand,
$$
d(\Phi_i\circ f)_{x_j}
=
d\left(\sum_{l=1}^4(\Phi_i)_{x_l}(f_l)_{x_j}\right)
=
\sum_{l=1}^4
 (\Phi_i)_{x_l}d\big((f_l)_{x_j}\big)
=
\sum_{l=1}^2
 (\Phi_i)_{x_l}d\big((f_l)_{x_j}\big)
$$
%$$
%\begin{array}{rcl}
%d(\Phi_i\circ f)_{x_j}
%&=&
%d\big(
%(\Phi_i)_{x_1}(f_1)_{x_j}+(\Phi_i)_{x_2}(f_2)_{x_j}
%+
%(\Phi_i)_{x_3}(f_3)_{x_j}+(\Phi_i)_{x_4}(f_4)_{x_j}\big)\\
%&=&
%(\Phi_i)_{x_1}d\big((f_1)_{x_j}\big)+(\Phi_i)_{x_2}d\big((f_2)_{x_j}\big)
%+
%(\Phi_i)_{x_3}d\big((f_3)_{x_j}\big)+(\Phi_i)_{x_4}d\big((f_4)_{x_j}\big),\\
%&=&
%(\Phi_i)_{x_1}d\big((f_1)_{x_j}\big)+(\Phi_i)_{x_2}d\big((f_2)_{x_j}\big),\\
%%
%d(\Phi_i\circ f)_{x_2}
%&=&
%(\Phi_i)_{x_1}d((f_1)_{x_2})+(\Phi_i)_{x_2}d((f_2)_{x_2})
%+
%(\Phi_i)_{x_3}d((f_3)_{x_2})+(\Phi_i)_{x_4}d((f_4)_{x_2}),\\
%&=&
%(\Phi_i)_{x_1}d((f_1)_{x_2})+(\Phi_i)_{x_2}d((f_2)_{x_2}),\\
%\end{array}
%\ (i,j=1,2)
%$$
$(i,j=1,2)$ hold at $0$. Thus 
$$
\begin{array}{l}
\Big(
d\big((\Phi_1\circ f)_{x_1}\big),\ 
d\big((\Phi_2\circ f)_{x_1}\big),\ 
d\big((\Phi_1\circ f)_{x_2}\big),\ 
d\big((\Phi_2\circ f)_{x_2}\big)
\Big)\\
\hspace{10mm}
=
\Big(
d\big((f_1)_{x_1}\big),\ 
d\big((f_2)_{x_1}\big),\ 
d\big((f_1)_{x_2}\big),\ 
d\big((f_2)_{x_2}\big)\Big)
\pmt{
J\Phi&O\\
O&J\Phi},
\end{array}
$$
where 
$$
J\Phi=
\pmt{(\Phi_1)_{x_1}&(\Phi_2)_{x_1}\\
(\Phi_1)_{x_2}&(\Phi_2)_{x_2}},\quad
O=\pmt{0&0\\0&0}
$$
holds at $0$.
This shows the desired result.
\end{proof}
Next we study the relation between
the condition (1) and
the quotient ring $Q(f)$.
Let $\E_n$ be the local ring of function-germs
$(\R^m,0)\to(\R,*)$.
Let $f=(f_1,\ldots,f_4)\in C^\infty(4,4)$ and
$Q(f)$ denote the quotient ring
$\E_4/\spann{f_1,f_2,f_3,f_4}_{\E_4}$.
\begin{lemma}
Let\/
$f\in C^\infty(4,4)$ 
satisfies that\/
$\rank df_{\zv}=2$.
Then\/
$Q(f)=\E_2/\spann{x_1^2,x_2^2}_{\E_2}$
$($respectively, 
$Q(f)=\E_2/\spann{x_1^2-x_2^2,x_1x_2}_{\E_2})$
is equivalent to that\/
$\det\hess_{(\xi,\eta)}\lambda(0)<0$
$($respecti\-vely, $\det\hess_{(\xi,\eta)}\lambda(0)>0)$ holds,
where\/ $(\xi,\eta)$ is a pair of vector fields
on the source such that\/ $\xi,\eta$ are
the basis of\/ $\ker df$ at\/ $0$.
\end{lemma}
\begin{proof}
Since the condition and the conclusion 
do not depend on the choice of coordinate systems
and the choice of vector fields,
we may assume that
\begin{equation}\label{eq:normalized}
f(x_1,x_2,x_3,x_4)
=(f_1(x_1,x_2,x_3,x_4),f_2(x_1,x_2,x_3,x_4),x_3,x_4),
\end{equation}
$(d f_1)_{\zv}=(d f_2)_{\zv}=\zv$,
$\xi=\partial x_1$ and $\eta=\partial x_2$.
Then
$$
Q(f)=\E_{2}
/
\spann{f_1(x_1,x_2,0,0),f_2(x_1,x_2,0,0)}_{\E_2}.
$$
Let us assume $Q(f)=\E_2/\spann{x_1^2,x_2^2}_{\E_2}$,
and let
$i$ be an isomorphism $i:Q(f)\to\E_2/\spann{x_1^2,x_2^2}_{\E_2}$.
Then the conclusion is obvious by
considering a coordinate change
$i(x_1),i(x_2)$ to $X_1,X_2$ on the 
target.
We show the converse.
We set the second order terms of
$f_1(x_1,x_2,0,0),f_2(x_1,x_2,0,0)$
as
\begin{equation}\label{eq:qf1}
a_1x_1^2+2b_1x_1x_2+c_1x_2^2,\ 
a_2x_1^2+2b_2x_1x_2+c_2x_2^2.
\end{equation}
Then we have
$$
\det\hess\lambda/16=
4 a_2 b_1 b_2 c_1 
- 4 a_1 b_2^2 c_1 - a_2^2 c_1^2 - 4 a_2 b_1^2 c_2 
+ 4 a_1 b_1 b_2 c_2 + 2 a_1 a_2 c_1 c_2 
- a_1^2 c_2^2.
$$
By a suitable coordinate change,
we may assume 
$b_1=0$ and $a_2=0$
in \eqref{eq:qf1}.
\mycomment{
By a coordinate change of $x_1$,
we may assume that
$$
\bar a_1x_1^2+\bar c_1x_2^2,\ 
\bar a_2x_1^2+2\bar b_2x_1x_2+\bar c_2x_2^2.
$$
By a coordinate change,
we may assume that 
$$
\bar a_1x_1^2+\bar c_1x_2^2,\ 
2\bar{\bar b}_2x_1x_2+\bar{\bar c}_2x_2^2.
$$
We rewrite $\bar a$ to $a$.
Then we have
$$
\bar a_1x_1^2+\bar c_1x_2^2,\ 
(2b_2x_1+c_2x_2)x_2.
$$
}
If $b_2\ne0$, then by the coordinate change
$\tilde x_1=2b_2x_1+c_2x_2, \tilde x_2=x_2$,
it holds that
$$
\begin{array}{rcl}
\displaystyle
&&\spann{f_1(x_1,x_2,0,0),f_2(x_1,x_2,0,0)}_{\E_2}\\
&=&
\spann{
\tilde a_1\tilde x_1^2+2\tilde b_1\tilde x_1\tilde x_2
+\tilde c_1\tilde x_2^2+O(3),\ 
\tilde x_1\tilde x_2+O(3)}_{\E_2}\\
&=&
\displaystyle
\spann{
\tilde a_1\tilde x_1^2
+\tilde c_1\tilde x_2^2+O(3),\ 
\tilde x_1\tilde x_2+O(3)}_{\E_2}
\end{array}
$$
for some $\tilde a_1(\ne0),\tilde b_1,\tilde c_1\in\R$,
where $O(3)$ means the terms which consist of terms
whose degrees are higher than 3.
In the case $\det\hess\lambda<0$,
since
$\det\hess\lambda=-64\tilde a_1\tilde c_1$,
we have $\tilde a_1\tilde c_1>0$.
Thus it holds that
$$
\spann{
\tilde a_1\tilde x_1^2
+\tilde c_1\tilde x_2^2,\ 
\tilde x_1\tilde x_2}_{\E_2}
=
\spann{
\tilde x_1^2
+\tilde x_2^2,\ 
\tilde x_1\tilde x_2}_{\E_2}
=
\spann{
\tilde x_1^2,\tilde x_2^2}_{\E_2},
$$
where we omit $O(3)$.
If $b_2=0$, since $a_1c_2\ne0$, it holds that
$$
\displaystyle
\spann{f_1(x_1,x_2,0,0),f_2(x_1,x_2,0,0)}_{\E_2}
=
\spann{
x_1^2+O(3),\ x_2^2+O(3)}_{\E_2}.
$$
Thus in both cases
$Q(f)=\E_2/\spann{x_1^2+O(3),x_2^2+O(3)}_{\E_2}$ holds.
One can show that this is isomorphic to
$\E_2/\spann{x_1^2,x_2^2}_{\E_2}$.

In the case, $\det\hess\lambda>0$,
we have $\tilde a_1\tilde c_1<0$.
Thus
$$
\displaystyle
\spann{f_1(x_1,x_2,0,0),f_2(x_1,x_2,0,0)}_{\E_2}
=
\spann{
x_1^2-x_2^2+O(3),\ x_1x_2+O(3)}_{\E_2}
$$
holds.
If $b_2=0$, then 
it holds that
$
\det\hess\lambda=-16a_1^2c^2_2.
$
This means that this case does not occur.
Hence
$Q(f)=\E_2/\spann{x_1^2-x_2^2+O(3),x_1x_2+O(3)}_{\E_2}$ holds.
Since $x_1^3,x_2^3\in \spann{x_1^2-x_2^2,x_1x_2}_{\E_2}$ holds,
this is isomorphic to
$\E_2/\spann{x_1^2-x_2^2,x_1x_2}_{\E_2}$.
\end{proof}

The 1-jet extension 
%$j^1f$ of the map-germ $f$ of the form \eqref{eq:normalized}
%is transverse to the set
$j^1f$\hfill of\hfill the\hfill map-germ\hfill $f$\hfill of\hfill the\hfill 
form\hfill \eqref{eq:normalized}\hfill is\hfill transverse\hfill 
to\hfill the\hfill set\\
 $\Sigma^2=\{j^1f\,|\,\rank (df)(0)=2\}$
at $0$
if and only if
$$
\det\Big(
d\big((f_1)_{x_1}\big),\ 
d\big((f_2)_{x_1}\big),\ 
d\big((f_1)_{x_2}\big),\ 
d\big((f_2)_{x_2}\big)
\Big)(0)\ne0,
$$
since 
$\Sigma^2=\big\{j^1f\,|\,
d\big((f_1)_{x_1}\big)=
d\big((f_2)_{x_1}\big)=
d\big((f_1)_{x_2}\big)=
d\big((f_2)_{x_2}\big)=0\big\}$.
Summarizing the above arguments,
and
by following the same arguments as in \cite[p183--186]{GG},
and taking care to use orientation-preserving 
diffeomorphism-germs,
one can complete the proof of Theorem \ref{thm:sigma20cris}.

Theorem \ref{thm:sigma20cris}
shows that 
there are two $\A$-isotopy classes in the
$\A$-class of $\Sigma^{2,0}_{\rm hyp}$,
and
there are four $\A$-isotopy classes in that 
of $\Sigma^{2,0}_{\rm elli}$.
The invariant $\ep_2$ is equal to
the mapping degree.
See \cite[Corollary 5.13]{df} for its global meaning.
See also \cite[Theorem 2.5, Remark 2.6]{et},
\cite[page 398]{p}, \cite[Theorem I]{s}.
As a corollary, we get an
$\A$-criteria 
for $\Sigma^{2,0}$-singularities:
\begin{corollary}
Let\/ $f\in C^\infty(4,4)$ be a map-germ such that\/
$\rank (df)(0)=2$ holds.
Then\/ $f$ is\/ $\A$-equivalent to\/ $\Sigma^{2,0}_{{\rm hyp}}$
$($respectively, $\Sigma^{2,0}_{{\rm elli}})$
if and only if
for a coordinate system\/ $(X_1,X_2,X_3,X_4)$ on the target
satisfying that\/
$d(X_1\circ f)(0)=d(X_2\circ f)(0)={0}$
and a pair of vector fields\/ $(\xi,\eta)$ on the source satisfying that\/
$\spann{\xi(0),\eta(0)}=\ker (df)(0)$,
it holds that
\begin{itemize}
\item[$(1)$] $\det\hess_{(\xi,\eta)}\lambda(\zv)<0$ 
$($respectively, $\det\hess_{(\xi,\eta)}\lambda(\zv)>0)$, and
\item[$(2)$] $\det(d(\xi g_1),d(\xi g_2),d(\eta g_1),d(\eta g_2))\ne0$.
\end{itemize}
\end{corollary}
\bigskip

{\bf Acknowledgment}
The author thanks professor Toru Ohmoto for fruitful discussion,
he suggested the application in Section \ref{sec:perturb}.
He also thanks to professor Stanis{\l}aw Janeczko
for constant encouragement.
He was
partly supported by Grant-in-Aid for Scientific Research
(C) No.\ 26400087,
from the Japan Society for the Promotion of Science.

%\section{thebibliography}

\begin{flushright}
\begin{tabular}{c}
Department of Mathematics,\\
Graduate School of Science, \\
Kobe University, \\
Rokko, Nada, Kobe 657-8501, Japan\\
{\tt sajiO$\!\!\!$amath.kobe-u.ac.jp}
\end{tabular}
\end{flushright}

\end{document}